\input ./style/arxiv-vmsta.cfg
\documentclass[numbers,compress,v1.0.1]{vmsta}

\volume{2}
\issue{2}
\pubyear{2015}
\firstpage{185}
\lastpage{198}
\doi{10.15559/15-VMSTA31}

\startlocaldefs
\newtheorem{theorem}{Theorem}
\newtheorem{lem}{Lemma}
\newtheorem{corollary}{Corollary}
\newtheorem{proposition}{Proposition}
\theoremstyle{definition}
\newtheorem{definition}{Definition}
\newtheorem{rem}{Remark}
\newtheorem*{remark}{Remark}
\newtheorem*{remarks}{Remarks}
\newtheorem*{examples}{Examples}

\allowdisplaybreaks
\endlocaldefs

\begin{document}
\begin{frontmatter}

\title{A group action on increasing sequences of set-indexed Brownian
motions}
\author{\inits{A.}\fnm{Arthur}\snm{Yosef}}\email{yusupoa@yahoo.com}
\address{Tel Aviv-Yaffo Academic College, 2 Rabeno Yeruham st., Tel
Aviv-Yaffo, Israel}

\markboth{A. Yosef}{A group action on increasing sequences of
set-indexed Brownian
motions}
\begin{abstract}
We prove that a square-integrable set-indexed stochastic process is a
set-indexed Brownian motion if and only if its projection on all the
strictly increasing continuous sequences are one-parameter
$G$-time-changed Brownian motions. In addition, we study the
``sequence-independent variation'' property for group
stationary-increment stochastic processes in general and for a
set-indexed Brownian motion in particular. We present some
applications.
\end{abstract}
\received{23 February 2015}
%
\revised{29 July 2015}
\accepted{30 July 2015}
\publishedonline{6 August 2015}

\end{frontmatter}

\section{Introduction}

The set-indexed Brownian motion $\{X_{A}:A\in\mathbf{A}\}$ is well
defined and well studied (see \cite{IvMe}). We will mention that the
indexing collection $\mathbf{A}$ is a compact set collection on a
topological space $T$. The choice of the collection $\mathbf{A}$ is
crucial: it must be sufficiently rich in order to generate the Borel
sets of $T$, but small enough to ensure the existence of a continuous
Gaussian process defined on $\mathbf{A}$.

In this paper, we define a group action on the indexing collection
$\mathbf{A}$, and from that we characterize the set-indexed Brownian
motion by using the notion of an increasing path introduced in
\cite{CaWa}. The characterization of a~set-indexed Brownian motion by
group action (Theorem \ref{thm1} in this article, which says that a
square-integrable set-indexed stochastic process is a set-indexed
Brownian motion if and only if its projection on all the strictly
increasing continuous sequences are one-parameter $G$-time-changed
Brownian motions) is the key to most of the proofs in this article. It
is of great importance since it allows us to ``divide and conquer.''
Therefore, many of the proofs for a set-indexed Brownian motion can be
recovered by reducing to a (classical) one-dimensional Brownian motion.
The results that we have extended from a classical Brownian motion to
a~set-indexed Brownian motion involve the following issues: hitting
time, maximum value, reflection principle, exiting from an interval,
time inversion, iterated logarithms, strong law of large numbers,
unboundedness, zero crossing, zero set, nondifferentiability,
path-independent variation, martingale in Brownian motion, and the
like.

The frame of a set-indexed Brownian motion is not only a new step of
generalization of a classical Brownian motion, but it proved a new look
upon a~Brownian motion. In recent years, there have been many new
results related to the dynamical properties of random processes indexed
by a class of sets. Set-indexed processes have many potential areas of
applications. For example: environment (increased occurrence of
polluted wells in a rural area could indicate a geographic region that
has been subjected to industrial waste), astronomy (a cluster of black
holes could be a result of an unobservable phenomenon affecting a
region in space), quality control (an increased rate of breakdowns in a
certain type of equipment might follow the failure of one or more
components), population health (unusually frequent outbreaks of a
disease such as leukemia near a nuclear power plant could signal a
region of possible air or ground contamination), and the like.

Cairoli and Walsh \cite{CaWa} introduced the notion of path-independent
variation (p.i.v) for two-parameter processes. They proved (under some
assumptions) that any strong martingale has the path-independent
variation property. We extend their results to set-indexed strong
martingales.

In the last section, we present some results concerning the
compensators of a set-indexed strong martingale and analyze the concept
of path-independent variation in connection with independent increments
in set-indexed process. We introduce compensators and demonstrate that
the path-independent variation property permits a better understanding
of the Doob--Meyer decomposition.

\section{Preliminaries}

We recall the definitions and notation from \cite{IvMe}.

\addtocounter{definition}{-1}
\begin{definition}\label{def0}
Let ($T,\tau$) be a nonvoid sigma-compact connected topological space.
A~nonempty class $\mathbf{A}$ of compact connected subsets of $T$ is
called an indexing collection if it satisfies the following:
\begin{enumerate}
\item[\rm(1)] $\varnothing\in\mathbf{A}$. In addition, there is an
    increasing sequence ($B_{n}$) of sets in $\mathbf{A}$ such that
    $T={\xch{\bigcup}{\cup}}_{n=1}^{\infty}B_{n}^{\circ}$.

\item[\rm(2)] $\mathbf{A}$ is closed under arbitrary intersections,
    and if $A,B\in\mathbf{A}$ are nonempty$,$ then $A\cap B$ is
    nonempty. If ($A_{i}$) is an increasing sequence in
    $\mathbf{A}$ and if there exists $n$ such that $A_{i}\subseteq
    B_{n}$ for every $i$, then $\overline{{\xch{\bigcup}{\cup}}
    _{i}A_{i}}\in\mathbf{A}$.

\item[\rm(3)] $\sigma(\mathbf{A})=\mathbf{B}$ where $\mathbf{B}$ is
    the collection of Borel sets of $T$.\vadjust{\goodbreak}
\end{enumerate}
\end{definition}

\begin{remarks}
(a) Note that any collection of sets closed under intersections is a
semilattice with respect to the partial order of the inclusion.

(b) Definition \ref{def0} implies that a space $T$ cannot be discrete
and that $\mathbf{A}$ is at least a continuum.
\end{remarks}

\begin{examples}
(a) The classical example is $T=\Re_{+}^{2}$ and
$\mathbf{A}=\mathbf{A}(\Re_{+}^{2})=\{[0,x]:x\in\Re_{+}^{2}\}$ (This
example can be extended to $T=\Re_{+}^{d}$ and
$\mathbf{A}(\Re_{+}^{d})=\{[0,x]:x\in\Re_{+}^{d}\}$, which will give
rise to a sort of $2^{d}$-sides process).

(b) The example (a) may be generalized as follows. Let $T=\Re_{+}^{2}$
and take $\mathbf{A}$ (or $\mathbf{A}(Ls))$ to be the class of compact
\textit{lower sets}, i.e. the class of compact subsets $A$ of $T$
satisfying $t\in A$ implies $[0,t]\subseteq A$.
\end{examples}

\begin{definition}\label{def1}
Let $(\varOmega,F,P)$ be a complete probability space equipped with an
$\mathbf{A}$-indexed filtration $\{F_{A}:A\in\mathbf{A}\}$ that
satisfies the following conditions:
\begin{enumerate}
\item[\rm(i)] for~all $A\in\mathbf{A}$, we have $F_{A}\subseteq F$,
    and $F_{A}$ contains the $P$-null sets.
\item[\rm(ii)] for all $A,B\in\mathbf{A}$, if $A\subseteq B$, then
    $F_{A}\subseteq F_{B}$.
\item[\rm(iii)] $F_{{\xch{\bigcap}{\cap}} A_{i}}={\xch{\bigcap}{\cap}} F_{A_{i}}$
    for any decreasing sequence $\{A_{i}\}$ in $\mathbf{A}$ (for
    consistency, in what follows, if $T\notin\mathbf{A}$, we define
    $F_{T}=F$).
\end{enumerate}
\end{definition}

We will need other classes of sets generated by $\mathbf{A}$. The first
is $\mathbf{A}(\mathbf{u})$, which is the class of finite unions of
sets in $\mathbf{A}$. We remark that $\mathbf{A}(\mathbf{u})$ is itself
a lattice with the partial order induced by set inclusion. Let
$\mathbf{C}$ consist of all the subsets of $T$ of the form
$C=A\backslash B$, $A\in\mathbf {A}$, $B\in\mathbf{A}(\mathbf{u})$. If
$C\in\mathbf{C(u)}\backslash \mathbf{A}$ ($\mathbf{C(u)}$ is the class
of finite unions of sets in $\mathbf {C})$, then we denote
\[
\mathbf{G}_{C}^{\ast}=\underset{A\in\mathbf
{A}(\mathbf{u}),A\cap C=\varnothing}{{\xch{\bigvee}{\vee}}}F_{A}.
\]

In addition, let $A^{\prime}$ be any finite subsemilattice of
$\mathbf{A}$ closed under intersection. For $A\in A^{\prime}$, define
the left neighborhood of $A$ in $A^{\prime}$ to be the set
$C_{A}=A\backslash{\xch{\bigcup}{\cup}}_{B\in A^{\prime},B\subset A}B$. We note that
${\xch{\bigcup}{\cup}}_{A\in A^{\prime}}A={\xch{\bigcup}{\cup}}_{A\in A^{\prime}}C_{A}$ and that
the latter union is disjoint. The sets in $A^{\prime}$ can always be
numbered in the following way: $A_{0}=\varnothing^{\prime}$
($\varnothing^{\prime}={\xch{\bigcap}{\cap}}_{n}{\xch{\bigcap}{\cap}}
_{A\in\mathbf{A}_{n},A\neq\varnothing}A$; note that $\varnothing
^{\prime}\neq\varnothing$), and given $A_{0},\dots,A_{i-1}$, we choose
$A_{i}$ to be any set in $A^{\prime}$ such that $A\subset A_{i}$
implies that $A=A_{j}$ for some $j=1,\dots,i-1$. Any such numbering
$A^{\prime}=\{A_{0},\dots,A_{k}\}$ will be called ``consistent with the
strong past'' (i.e., if $C_{i}$ is the left neighborhood of $A_{i}$ in
$A^{\prime}$, then $C_{i}={\xch{\bigcup}{\cup}}
_{j=0}^{i}A_{j}\backslash{\xch{\bigcup}{\cup}}_{j=0}^{i-1}A_{j}$ and $A_{j}\cap
C_{i}=\varnothing$ for all $j=0,\dots,i-1$, $i=1,2,\dots,k$).

Any $\mathbf{A}$-indexed function that has a (finitely)
additive extension to $\mathbf{C}$ will be called additive
(and is easily seen to be additive on $\mathbf{C(u)}$ as well). For
stochastic processes, we do not necessarily require that each sample
path be additive, but the additivity will be imposed in an almost sure
sense:
\begin{definition}\label{def2}
A set-indexed stochastic process $X=\{X_{A}:A\in\mathbf{A}\}$ is
additive if it has an (almost sure) additive extension to $\mathbf{C}$:
$X_{\varnothing}=0$, and if $C,C_{1},C_{2}\in\mathbf {C}$ with
$C=C_{1}\cup C_{2}$ and $C_{1}\cap C_{2}=\varnothing$, then almost
surely $X_{C_{1}}+X_{C_{2}}=X_{C}$. In particular, if $C\in\mathbf{C}$
and $C=A\backslash{\xch{\bigcup}{\cup}}_{i=1}^{n}A_{i}$, $A,A_{1},\dots,A_{n}\in
\mathbf{A}$, then almost surely $X_{C}=X_{A}-\sum_{i=1}^{n}X_{A\cap
A_{i}}+\sum_{i<j}X_{A\cap A_{i}\cap A_{j}}-\dots+(-1)^{n}X_{A\cap
{\xch{\bigcap}{\cap}}_{i=1}^{n}A_{i}}$.
\end{definition}

We shall always assume that our stochastic processes
are additive. We note that a process with an (almost sure) additive
extension to $\mathbf{C}$ also has an (almost sure) additive extension
to $\mathbf{C(u)}$.

\begin{definition}\label{def3}
Let $(G,\cdot)$ be a group. The group $G$ will be called a permutation
group on $[a,b]$ if $G=\{\pi :[a,b]\rightarrow\lbrack a,b] \mid\pi$ is
a one-to-one and onto function$\}$, and we denote this group by
$S_{[a,b]}$ (i.e., $S_{[a,b]}$ is the class of all the bijection
functions from $[a,b]$ to $[a,b]$).
\end{definition}

\begin{definition}\label{def4}
A positive measure $\sigma$ on ($T,\mathbf{B}$) is called strictly
monotone on~$\mathbf{A}$ if $\sigma _{\varnothing ^{\prime}}=0$ and
$\sigma_{A}<\sigma_{B}$ for all $A,B\in\mathbf {A}$ such that
$A\subsetneq B$. The collection of these measures is denoted by
$M(\mathbf{A})$.
\end{definition}

The classical examples for this definition are the Lebesgue measure or
Radon measure when $T=\Re_{+}^{d}$ and
$\mathbf{A}=\mathbf{A}(\Re_{+}^{d})$.

\begin{definition}\label{def5}
Let $\sigma\in M(\mathbf{A})$, and let $(G,\cdot)$ be a group. A group
action $\ast$ of $(G,\cdot)$ on $\mathbf{A}$ is defined by $g\ast(A\cup
B)=g\ast A\cup g\ast B$, $g\ast(A\backslash B)=g\ast A\backslash g\ast
B$ for all $A,B\in\mathbf{A}$ and $g\in G$, and there exists
$\eta:G\rightarrow\Re_{+}$ such that $\sigma(g\ast A)=\eta(g)\sigma
(A)$ for all $A\in\mathbf{A}$ and $g\in G$.
\end{definition}

The classical examples are the following:
\begin{enumerate}
\item[\rm(a)] Let $G=(\Re_{+}^{d},\cdot)$ and
    $\mathbf{A}=\mathbf{A}(\Re_{+}^{d})=\{[0,x]:x\in\Re_{+}^{d}\}$.
    Then a group action is defined by $g\ast\lbrack0,t]=[0,g\cdot
    t]=[0,g_{1}t_{1}]\times\lbrack0,g_{2}t_{2}]\times\cdots\times
    \lbrack 0,g_{d}t_{d}]$, $\sigma(g\ast\lbrack
    0,t])=g_{1}g_{2}\ldots g_{n}\sigma([0,t])$ for all
    $g=(g_{1},\dots ,g_{d})\in G$ and
    $t=(t_{1},\dots,t_{d})\in\mathbf{\Re}_{+}^{d}$.
\item[\rm(b)] Let $G=(S_{[0,\infty)},\circ)$ and
    $\mathbf{A}=\mathbf{A}(\Re_{+}^{d})=\{[0,x]:x\in\Re_{+}^{d}\}$.
    Then a group action is defined by
    $\pi\ast\lbrack0,t]=[0,\pi\circ t_{1}]\times\cdots\times[0,\pi
    \circ t_{d}]$ for all $\pi\in S_{[0,\infty)}$ and
    $t=(t_{1},\dots,t_{d})\in\mathbf{\Re}_{+}^{d}$.
\end{enumerate}

\begin{definition}\label{def6}
Let $I\subseteq\Re$, and let $A=\{A_{\alpha}\}_{\alpha\in I}$ be
increasing sequence in $\mathbf{A(u)}$.
\begin{enumerate}
\item[\rm(a)] The sequence $A$ is called ``strictly increasing'' if
    $A_{\alpha}\subsetneq A_{\beta}$ for all $\alpha,\beta\in I$
    such that $\alpha<\beta$.
\item[\rm(b)] If $I=[a,b]$, then the sequence $A$ is called a
    ``continuous sequence'' if
    $A_{s}=\overline{{\xch{\bigcup}{\cup}}_{u<s}A_{u}}={\xch{\bigcap}{\cap}}_{v>s}A_{v}$ for
    all $s\in(a,b)$ and $A_{a}={\xch{\bigcap}{\cap}}_{v>a}A_{v}$,
    $A_{b}=\overline{{\xch{\bigcup}{\cup}} _{u<b}A_{u}}$.
\end{enumerate}

Given a set-indexed stochastic process $X$ and increasing sequence
$\{A_{\alpha}\}_{\alpha\in\lbrack a,b]}$ in $\mathbf{A(u)}$, we define
a process $Y$ indexed by $[a,b]$ as follows:
$Y_{s}=X_{A_{s}}=X_{s}^{A}$ for all $s\in\lbrack a,b]$.

A set-indexed stochastic process $X$ is called outer-continuous if $X$
is finitely additive on $\mathbf{C}$ and for any decreasing sequence
$\{A_{n}\}\in\mathbf{A}$, $X_{{\xch{\bigcap}{\cap}}_{n}A_{n}}=\lim_{n}X_{A_{n}}$ and
is called inner-continuous if for any increasing sequence $\{A_{n}\}\in
\mathbf{A}$ such that $\overline{{\xch{\bigcup}{\cup}}_{n}A_{n}}=A\in\mathbf{A}$,
$X_{A}=\lim_{n}X_{A_{n}}$.
\end{definition}

\begin{lem}[\cite{IvMe}]\label{lem1}
Let $A^{\prime}=\{\varnothing^{\prime}=A_{0},\dots,A_{k}\}$ be any
finite subsemilattice of $\mathbf{A}$ equipped with a numbering
consistent with the strong past. Then there exists a continuous
(strictly) flow $f:[0,k]\rightarrow \mathbf{A(u)}$ such that the
following are satisfied:
\begin{enumerate}
\item[\rm(i)] $f(0)=\varnothing^{\prime}$, $f(k)={\xch{\bigcup}{\cup}}_{j=0}^{k}A_{j}$.\vadjust{\eject}
\item[\rm(ii)] Each left neighborhood $C$ generated by $A^{\prime}$ is
of the form $C=f(i)\backslash f(i-1)$, $1\leq i\leq k$.
\item[\rm(iii)] If $C=f(t)\backslash f(s)$, then $C\in\mathbf{C(u)}$
and $F_{f(s)}\subseteq\mathbf{G}_{C}^{\ast}$ (for the definition of
a~continuous flow,
see \cite{IvMe}).
\end{enumerate}
\end{lem}

\begin{lem}\label{lem2}
Let $A^{\prime}=\{\varnothing ^{\prime}=A_{0},\dots,A_{k}\}$ be any
finite subsemilattice of $\mathbf{A}$ equipped with a numbering
consistent with the strong past.
\begin{enumerate}
\item[\rm(a)] Then there exists a strictly increasing and
    continuous sequence
    \[
    B^{(k)}=\bigl\{B_{\alpha}^{(k)}\bigr\}_{\alpha\in\lbrack0,k]}
    \]
    in $\mathbf{A(u)}$ such that the following are satisfied:
\begin{enumerate}
\item[\rm(i)] $B_{0}^{(k)}=\varnothing^{\prime}$,
    $B_{k}^{(k)}={\xch{\bigcup}{\cup}}_{j=0}^{k}A_{j}$.
\item[\rm(ii)] Each left neighborhood $C$ generated by
    $A^{\prime}$ is of the form $C=B_{i}^{(k)}\backslash
    B_{i-1}^{(k)}$, $1\leq i\leq k$.
\item[\rm(iii)] If $C=B_{t}^{(k)}\backslash B_{s}^{(k)}$, then
    $C\in\mathbf{C(u)}$ and $F_{B_{s}^{(k)}}\subseteq
    \mathbf{G}_{C}^{\ast}$.
\end{enumerate}
\item[\rm(b)] Then there exists a strictly increasing and
    continuous sequence
    \[
    B=\{B_{\alpha}\}_{\alpha\in\lbrack0,\infty)}
    \]
    in $\mathbf{A(u)}$ such that the following are satisfied:
\begin{enumerate}
\item[\rm(i)] $B_{0}=\varnothing^{\prime}$, $B_{k}={\xch{\bigcup}{\cup}}_{j=0}^{k}A_{j}$.
\item[\rm(ii)] Each left neighborhood $C$ generated by
    $A^{\prime}$ is of the form $C=B_{i}\backslash B_{i-1}$,
    $1\leq i\leq k$.
\item[\rm(iii)] If $C=B_{t}\backslash B_{s}$, then
    $C\in\mathbf{C(u)}$ and
    $F_{B_{s}}\subseteq\mathbf{G}_{C}^{\ast}$.
\end{enumerate}
\end{enumerate}
\end{lem}

\begin{proof}
(a) It is clear from Lemma \ref{lem1} by setting $B_{i}^{(k)}=f(i)$,
$1\leq i\leq k$.

(b) Notice that for each $k$, $B^{(k)}=B^{(k+1)}$ on
    $[0,k]$. Then we can define the sequence $B=\{B_{\alpha
    }\}_{\alpha\in\lbrack0,\infty)} $ in $\mathbf{A(u)}$ by
    $B_{\alpha}=B_{\alpha}^{(\left[ \alpha\right] +1)} $ for all
    $\alpha$.
\end{proof}

\begin{rem}\label{rem1}
Similarly to the construction performed in
Lemma 2, we can prove that for all increasing sequences
$\{B_{n}\}_{n=1}^{\infty}\in\mathbf{A(u)} $, there exists a strictly
increasing and continuous sequence $\{A_{\alpha
}\}_{\alpha\in\lbrack0,\infty)}$ in $\mathbf{A(u)}$ such that $A_{n}=B_{n}$.
\end{rem}

\section{A characterization of a set-indexed Brownian motion by sequences}

\begin{definition}\label{def7}
Let $\sigma\in M(\mathbf{A})$. We say that an $\mathbf{A}$-indexed process $X$ is a~Brownian motion with variance
$\sigma$ if $X$ can be extended to a finitely additive process on
$\mathbf{C(u)}$ and if for disjoint sets
$C_{1},\dots,C_{n}\in\mathbf{C}$, $X_{C_{1}},\dots,X_{C_{n}}$ are
independent zero-mean Gaussian random variables with variances
$\sigma_{C_{1}},\dots,\sigma_{C_{n}}$, respectively.

For any $\sigma\in M(\mathbf{A})$, there exists a set-indexed Brownian
motion with variance $\sigma$~\cite{IvMe}.
\end{definition}

\begin{definition}\label{def8}
(a) Let $X=\{X_{t}:t\geq0\}$ be a stochastic process, and let $\ast$ be
a group action of $(G,\cdot)$ on $\mathbf{\Re}_{+}$. The process $X$ is
called a $G$-time-changed Brownian motion if there exists $g\in G$ such
that $X^{g}=\{X_{g\ast t}:t\geq0\}$ is a~Brownian motion.\vadjust{\goodbreak}

(b) Let $X=\{X_{A}:A\in\mathbf{A}\}$ be a set-indexed process,
$\{A_{\alpha}\}_{\alpha\in\lbrack0,\infty)}$ be an increasing sequence
in $\mathbf{A(u)}$, and $\ast$ be a group action of $(S_{[0,\infty
)},\circ)$ on $\mathbf{\Re}_{+}$. The process $X^{A}$ (see Definition
\ref{def6}) is called a $G$-time-changed Brownian motion if there
exists $\pi\in S_{[0,\infty )}$ such that $X^{\pi,A}=\{X_{\pi\ast
A_{\alpha}}:\alpha\in\lbrack
0,\infty)\}=\{X_{A_{\pi(\alpha)}}:\alpha\in\lbrack0,\infty)\}$ is a
Brownian motion.
\end{definition}
The characterization of a set-indexed Brownian motion by a group action
on a sequence (Theorem \ref{thm1}) is very important and is the key to
most of the proofs in this part of the paper. It is of great importance
since it allows us to ``divide and conquer.'' Therefore, many
properties of a set-indexed Brownian motion can be recovered by
reducing them to a (classical) one-dimensional Brownian motion. Theorem
\ref{thm1} further says that a square-integrable set-indexed stochastic
process is a set-indexed Brownian motion if and only if its projections
on all the strictly increasing continuous sequences by a group action
are one-parameter time-changed Brownian motions.
\begin{theorem}[Characterization of a set-indexed Brownian motion by a
group action on sequences]\label{thm1}
Let $X=\{X_{A}:A\in\mathbf{A}\}$ be a square-integrable set-indexed
stochastic process. Suppose that
there exists a
group action $\ast$ of $(S_{[0,\infty)},\circ)$ on $\mathbf{A}$.
Let $\sigma\in M(\mathbf{A})$. Then $X$ is a set-indexed Brownian motion
with variance $\sigma$ if and only if  the process
$X^{A}=\{X_{A_{\alpha}}:\alpha\in\lbrack0,\infty)\}$ is an
$S_{[0,\infty)}$-time-changed
Brownian motion for all strictly increasing and continuous sequences $\{A_{\alpha}\}_{\alpha\in\lbrack0,\infty)}$ in $\mathbf{A(u)}$. In
other words (by Definition \ref{def8}), for all strictly increasing and
continuous sequences $\{A_{\alpha}\}_{\alpha\in\lbrack0,\infty)}$
in $\mathbf{A(u)}$, there exists $\pi\in S_{[0,\infty)}$ such that
$X^{\pi,A}=\{X_{\pi\ast A_{\alpha}}:\alpha\in\lbrack0,\infty
)\}=\{X_{A_{\pi(\alpha)}}:\alpha\in\lbrack0,\infty)\}$ is a
Brownian motion.
\end{theorem}
\vspace*{-6pt}
\begin{proof}({\it if})
Suppose that $X$ is a set-indexed Brownian motion with variance~$\sigma
$. Define $\theta:\Re_{+}\rightarrow\Re_{+}$by
$\theta(\alpha)=\sigma(A_{\alpha})$ ,$\alpha\in\lbrack0,\infty )$. The
function $\theta$ is strictly increasing and continuous because $A$ is
strictly increasing and continuous. Since $\sigma\in M(\mathbf{A})$,
$\theta$ is invertible. Let $\pi(\alpha)=\theta ^{-1}(\alpha)$; $\pi$
is continuous, and $\sigma(A_{\pi(\alpha )})=\alpha$. Then $\pi\in
S_{[0,\infty)}$, and  $X^{\pi ,A}=\{X_{\pi\ast
A_{\alpha}}:\alpha\in\lbrack0,\infty
)\}=\{X_{A_{\pi(\alpha)}}:\alpha\in\lbrack0,\infty)\}$ is a Brownian
motion.

({\it only if}) Suppose that for all strictly continuous sequences
$\{A_{\alpha}\}_{\alpha\in\lbrack0,\infty)}$ in $\mathbf{A(u)}$, there
exists $\pi\in S_{[0,\infty)}$ such that $X^{\pi,A}$ is a Brownian
motion. It must be shown that if $\{C_{1},\dots,C_{k}\}\in\mathbf{C}$
are disjoint, then $X_{C_{1}},\dots,X_{C_{k}}$ are independent normal
random with variances $\sigma(C_{1}),\dots,\sigma(C_{k})$,
respectively. Without loss of generality, we may assume that the sets
$\{C_{1},\dots,C_{k}\}$ are the left neighborhoods of the
subsemilattice $A^{\prime}$ of $\mathbf{A}$ equipped with a numbering
consistent with the strong past. By Lemma 2 there exists a strictly
increasing and continuous sequence
$\{A_{\alpha}\}_{\alpha\in\lbrack0,\infty)}$ in $\mathbf{A(u)}$ such
that each left neighborhood generated by $A^{\prime}$ is of the form
$C_{i}=A_{i}\backslash A_{i-1}$, $1\leq i\leq k$. Thus, $X^{A}$ is an
$S_{[0,\infty ]}$-time-changed Brownian motion such that $X_{C_{i}}
=X_{A_{i}}-X_{A_{i-1}}$ and $\sigma(C_{i})=\sigma(A_{i})-\sigma
(A_{i-1})$; therefore, $X_{C_{1}},\dots,X_{C_{k}}$ are independent
normal random with variances $\sigma(C_{1}),\dots,\sigma(C_{k})$,
respectively.
\end{proof}\vspace*{-6pt}
\begin{corollary}\label{cor1}
Let $X=\{X_{A}:A\in\mathbf{A}\}$ be a square-integrable set-indexed stochastic process
with $X_{\varnothing^{\prime}}=0$ that is inner- and outer-continuous. Let
$\sigma\in M(\mathbf{A})$. Then $X$ is a set-indexed Brownian motion with variance
$\sigma$ if and only if for all strictly continuous sequences
$\{A_{\alpha}\}_{\alpha\in\lbrack0,\infty)}$ in
$\mathbf{A}(\mathbf{u})$, the process $X^{A}$ has independent
increments\vadjust{\eject} and there exists $\pi\in S_{[0,\infty)}$ such that
$X^{\pi,A}=\{X_{A_{\pi(\alpha)}}:\alpha\in\lbrack0,\infty)\}$
has stationary increments. (The definition and more details about
independent increments and stationary increments can be found in
\cite{IvMe}.)
\end{corollary}

\begin{proof}({\it if}) Obvious.

({\it only if}) Suppose that for all strictly continuous sequences
$\{A_{\alpha}\}_{\alpha\in\lbrack0,\infty)}$ in
$\mathbf{A}(\mathbf{u})$, the process $X^{A}$ has independent
increments and there exists $\pi\in S_{[0,\infty)}$ such that
$X^{\pi,A}=\{X_{A_{\pi (\alpha)}}:\alpha\in\lbrack0,\infty)\}$ has
stationary increments. Since $X$ is inner- and outer-continuous,
$X^{A}$ is continuous (see \cite{IvMe}). The process $X^{A}$ has
independent increments, and there exists $\pi\in S_{[0,\infty)}$ such
that $X^{\pi,A}$ has stationary increments; therefore, $X^{A}$ is an
$S_{[0,\infty)}$-time-changed Brownian motion for all strictly
continuous sequences $\{A_{\alpha}\}_{\alpha\in\lbrack0,\infty)}$ in
$\mathbf{A}(\mathbf{u})$. Thus, from Theorem~\ref{thm1} we conclude
that $X$ is a set-indexed Brownian motion with variance~$\sigma$.
\end{proof}

\begin{definition}[\cite{IvMe}] \label{def9} Let $X=\{X_{A}:A\in
\mathbf{A}\}$ be an integrable additive set-indexed stochastic process
adapted with respect to filtration $F=\{F_{A}:A\in\mathbf{A}\}$. The
process $X$  is said to be:
\begin{enumerate}
\item A $\mathbf{C}$-strong martingale (or in short notation, a
    strong martingale) if for all \mbox{$C\in\mathbf{C}$}, we have
    $E[X_{C}|\mathbf{G}_{C}^{\ast}]=0$;

\item A martingale if for any $A,B\in\mathbf{A}$ such that
$A\subseteq B$, we have $E[X_{B}|F_{A}]=X_{A}$.
\end{enumerate}
\end{definition}

For studies of different kinds of martingales, see \cite{Kh,MeNu,Za}.

In particular, if $T=\Re_{+}^{2}$ and $\mathbf{A}=\mathbf{A}(\Re
_{+}^{2})$ then $X$ is said to be a strong
martingale-$\mathbf{\Re}_{+}^{2}$ if $X$ is adapted, vanishes on the
axes, and $E[X((z,z^{\prime}])|F_{z}^{1}\vee F_{z}^{2}]=0$ for all
$z\leq z^{\prime}$, where $[z,z^{\prime }]=[s,s^{\prime}]\times\lbrack
t,t^{\prime}]$, $F_{z}^{1}={{\xch{\bigvee}{\vee}}}_{v}F_{sv}$,
$F_{z}^{2}={{\xch{\bigvee}{\vee}}}_{u}F_{ut}$,
$z=(s,t),z^{\prime}=(s^{\prime},t^{\prime})$. (This definition and
additional explanation can be found in \cite{CaWa}).

\begin{remark}
Under some hypotheses, we can define $\langle X\rangle$ to be the
compensator associated with the submartingale $X^{2}$. The definition
and more details regarding $\langle X\rangle$ can be found in
\cite{IvMe,CaWa}.
\end{remark}

From the well-known L\'evy martingale characterization
of the Brownian motion (see \cite{Du} or \cite{ReYo}) we get the following
corollary.

\begin{corollary}\label{cor2}
Let $X=\{X_{A}:A\in \mathbf{A}\}$ be a square-integrable set-indexed
martingale with $X_{\varnothing^{\prime}}=0$ that is inner- and
outer-continuous. Let $\sigma\in M(\mathbf{A})$. Then $X$ is a
set-indexed Brownian motion with variance $\sigma$ if and only if
$\langle X^{A}\rangle$ is deterministic for all strictly increasing and
continuous sequence $\{A_{\alpha}\}_{\alpha\in\lbrack0,\infty)}$ in
$\mathbf{A}(\mathbf{u})$.
\end{corollary}

\begin{proof}({\it if})
Suppose that $X$ is a set-indexed Brownian motion with
variance~$\sigma$. By Theorem \ref{thm1} the process $X^{A}$ is an
$S_{[0,\infty)} $-time-changed Brownian motion for all strictly
increasing and continuous sequences
$\{A_{\alpha}\}_{\alpha\in\lbrack0,\infty)}$ in $\mathbf{A(u)}$. Then
from the L\'evy characterization we conclude that $\langle
X^{A}\rangle$ is deterministic for all strictly increasing and
continuous sequences $\{A_{\alpha}\}_{\alpha\in\lbrack0,\infty)}$ in
$\mathbf{A}(\mathbf{u})$.

({\it only if}) Suppose that $\langle X^{A}\rangle$ is deterministic
for all strictly increasing and continuous sequences
$\{A_{\alpha}\}_{\alpha\in\lbrack0,\infty)}$ in
$\mathbf{A}(\mathbf{u})$. Since $X$ is inner- and outer-continuous,
$X^{A}$ is continuous (see \cite{IvMe}). Since $X$ is a set-indexed
martingale, $X^{A}$is a~martingale. But if the process $X^{A}$ is a
martingale and $\langle X^{A}\rangle$ is deterministic, then based on
the L\'evy characterization we get that $X^{A}$ is an
$S_{[0,\infty)}$-time-changed Brownian motion for all strictly
increasing and continuous sequences $\{A_{\alpha}\}_{\alpha
\in\lbrack0,\infty)}$ in $\mathbf{A}(\mathbf{u})$. Thus, from Theorem
\ref{thm1} we conclude that $X$ is a set-indexed Brownian motion with
variance $\sigma$.
\end{proof}

In addition, Theorem \ref{thm1} is an important ``bridge'' from a
set-indexed Brownian motion to a Brownian motion, and from that we
extend many theorems, such as hitting time (Corollary \ref{cor3}),
reflection principle (Corollary \ref{cor4}), exiting from an interval
(Corollary \ref{cor5}), unboundedness (Corollary \ref{cor6}), strong
law of large numbers (Corollary \ref{cor7}), law of iterated logarithm
(Corollary \ref{cor8}), the zero set (Corollary~\ref{cor9}), and the
like.

Let $L$ be a decreasing continuous line in $\Re_{+}^{2}$. If
$A\in\mathbf{A}(\Re_{+}^{2})$, then we
\begin{enumerate}
\item[\rm(a)] write $A\prec L$ ($A\succ L$) if there exist
    $(x,y)\in L$ such that $A\subset\lbrack0,x]\times\lbrack 0,y]$
    ($A\supset\lbrack0,x]\times\lbrack0,y]$);

\item[\rm(b)] write $A\preceq L$ ($A\succeq L$) if there exist
    $(x,y)\in L$ such that $A\subseteq\lbrack0,x]\times\lbrack
    0,y]$ ($A\supseteq \lbrack0,x]\times\lbrack0,y]$).

\item[\rm(c)] write $A\in L$ if there exist $(x,y)\in L$ such that
    $A=[0,x]\times\lbrack0,y]$.
\end{enumerate}

Let $X=\{X_{A}:A\in\mathbf{A}(\Re_{+}^{2})\}$ be a set-indexed Brownian
motion. For $a>0$, we define $L_{a}$ to be a decreasing continuous line
in $\Re_{+}^{2}$ such that
\begin{enumerate}
\item[\rm(a)] if $A\prec L_{a}$, then $X_{A}<a$, and

\item[\rm(b)] if $A\in L_{a}$, then $X_{A}\geq a$ for the first time on $A$.
\end{enumerate}
(In other words, $L_{a}$ is the collection of points $(x,y)$ when $X$
reaches the value $a$ for the first time.)

\begin{corollary}[Hitting time]\label{cor3}
Let $X=\{X_{A}:A\in\mathbf{A}(\Re_{+}^{2})\}$ be a set-indexed Brownian
motion with variance $\sigma$ (Lebesgue measure). Then
\[
P[L_{a}\preceq A]=2-2\varPhi \biggl(\frac{a}{\sqrt{\sigma_{A}}}\biggr)\quad \mbox{for all}\
A\in\mathbf{A}\bigl(\Re_{+}^{2}\bigr)
\]
($\varPhi$ is the standard Gaussian distribution function).
\end{corollary}

\begin{proof}
Let $A\in\mathbf{A}(\Re_{+}^{2})$. Then $P[X_{A}\geq a]=P[X_{A}\geq
a|A\prec L_{a}]P[A\prec L_{a}]+P[X_{A}\geq a|A\succeq L_{a}]P[A\succeq
L_{a}]$. From the definition of $L_{a}$ we conclude that if $A\prec
L_{a}$, then $X_{A}<a$, and thus $P[X_{A}\geq a|A\prec L_{a}]=0$. It is
clear that there exist a strictly increasing and continuous sequence
$\{B_{\alpha }\}_{\alpha\in \lbrack0,\infty)}$ in $\mathbf{A(u)}$ and
$\alpha_{0}\geq0$ such that $B_{\alpha_{0}}=$ $A$. The sequence
$B=\{B_{\alpha}\}_{\alpha\in \lbrack0,\infty)}$ is strictly continuous;
therefore, from Theorem \ref{thm1} we conclude that $X^{B}$ is a
$G$-time-changed Brownian motion. (In other words, there exists $\pi
\in S_{[0,\infty)}$ such that $X^{\pi,B}$ is a Brownian motion.) By the
symmetry of $X^{\pi,B}$ it is clear that $P[X_{A}\geq a|A\succeq
L_{a}]=P[X_{\alpha _{0}}^{\pi,B}\geq a|A\succeq
L_{a}]=P[X_{B_{\pi(\alpha_{0})}}\geq a|A\succeq L_{a}]=\frac{1}{2}$,
and thus $P[L_{a}\preceq A]=2P[X_{A}\geq
a]=2-2\varPhi(\frac{a}{\sqrt{\sigma_{A}}})$.\vadjust{\goodbreak}
\end{proof}

\begin{corollary}[Reflection principle]\label{cor4}
Let $X=\{X_{A}:A\in\mathbf{A}(\Re_{+}^{2})\}$ be a set-indexed Brownian
motion with variance $\sigma$ (Lebesgue measure). Then, for
$A\in\mathbf{A}$,
\[
W_{A}=\left\{
\begin{array}{cc}
X_{A}, & A\prec L_{a} \\
2a-X_{A}, & A\succeq L_{a}
\end{array}
\right\}
\]
is a set-indexed Brownian motion with variance $\sigma$.
\end{corollary}

\begin{proof}
We must show that if $\{C_{1},\dots,C_{k}\}$
are disjoint, then $W_{C_{1}},\dots,W_{C_{k}}$ are independent normal
random variables with variances $\sigma_{C_{1}},\dots,\sigma
_{C_{k}}$, respectively. Similarly to the construction done in Theorem
\ref{thm1}, we can get that there exists a strictly increasing and continuous
sequence $\{A_{\alpha}\}_{\alpha\in\lbrack0,\infty)}$ in
$\mathbf{A(u)}$ such that $C_{i}=A_{i}\backslash A_{i-1}$. The sequence
$\{A_{\alpha}\}_{\alpha\in\lbrack0,\infty)}$ is strictly
continuous; therefore, from Theorem \ref{thm1} we conclude that there exists
$\pi\in S_{[0,\infty)}$ such that $X^{\pi,A}$ is a Brownian motion.
Clearly, there exists $\alpha_{a}\geq0$ such that $A_{\pi(\alpha
_{a})}\in L_{a}$. We recall
that if $X=\{X_{t}:t\geq0\}$ is be a classical Brownian motion and
$T_{a}=\inf\{t\geq0: X_{t}=a\}$, then
\[
Z_{t}=\left\{
\begin{array}{cc}
X_{t}, & t<T_{a} \\
2a-X_{t}, & t\geq T_{a}
\end{array}
\right\}\quad (t\geq0)
\]
is a Brownian motion \cite{Fr}. Thus, if we define
\[
W_{\alpha}=\left\{
\begin{array}{cc}
X_{\alpha}^{\pi,A}, & \alpha<\alpha_{a} \\
2a-X_{\alpha}^{\pi,A}, & \alpha\geq\alpha_{a}%
\end{array}
\right\},
\]
then $W_{\alpha}$ turns out to be a Brownian motion, and so
$W_{C_{1}},\dots,W_{C_{k}}$ are independent normal random variables
with variances $\sigma_{C_{1}},\dots,\sigma_{C_{k}}$, respectively.
\end{proof}

Let $X=\{X_{A}:A\in\mathbf{A}(\Re_{+}^{2})\}$ be a
set-indexed Brownian motion. For $a<0<b$, define $D_{(a,b)}=\{A\in
\mathbf{A}: X_{A}\notin(a,b)\ \mbox{for the first time}\}$. (In other
words, $D_{(a,b)}$ is the
collection of sets $A\in\mathbf{A}$ such that if $A\in D_{(a,b)}$,
then $X_{A}\notin(a,b)$ for the first time on $A$).

\begin{corollary}[Reflection principle]\label{cor5}
Let $X=\{X_{A}:A\in\mathbf{A}(\Re_{+}^{2})\}$ be a set-indexed Brownian
motion with variance $\sigma$ (Lebesgue measure). If
$T(a,b)={\xch{\bigcap}{\cap}}_{A\in D(a,b)}A$, then
$P[X_{T(a,b)}=b]=\frac{|a|}{b+|a|}$.
\end{corollary}

\begin{proof}
It is clear that $T(a,b)\in\mathbf{A}(\Re_{+}^{2})$. It is easy to see
that there exists a strictly increasing and continuous sequence
$\{A_{\alpha}\}_{\alpha\in\lbrack0,\infty)}$ in
$\mathbf{A}(\mathbf{u})$ and there exists $0\leq\beta$ such that
$A_{\beta}=$ $T(a,b)$. The sequence $\{A_{\alpha}\}_{\alpha\in
\lbrack0,\infty)}$ is strictly increasing and continuous; therefore,
from Theorem \ref{thm1} we conclude that $X^{A}$ is an $S_{[0,\infty
)}$-time-changed Brownian motion. (In other words, there exists a $\pi
\in S_{[0,\infty)}$ such that $X^{\pi,A}=\{X_{A_{\pi(\alpha
)}}:\alpha\in\lbrack0,\infty)\}$ is a Brownian motion, and there exists
$0\leq\alpha_{(a,b)}$ such that $A_{\beta}=A_{\pi(\alpha
_{(a,b)})}=T(a,b) $). We recall that if $X=\{X_{t}:t\geq0\}$ is a
Brownian motion and $t(a,b):=\inf_{t\geq0}\{X_{t}\notin (a,b)\}$ for
$a<0<b$, then $P[X_{t(a,b)}=b]=\frac{|a|}{b+|a|}$. Thus,
$P[X_{T(a,b)}=b]=P[X_{\alpha_{(a,b)}}^{\pi,A}=b]=\frac{|a|}{b+|a|}$.
\end{proof}

\begin{corollary}[Unboundedness]\label{cor6}
Let $\sigma\in M(\mathbf{A})$, and let $X=\{X_{A}:A\in\mathbf{A}\}$ be
a set-indexed Brownian motion with variance $\sigma$. Then ${\sup
}_{A\in\mathbf{A}}X_{A}(\omega )=+\infty$ and
$\inf_{A\in\mathbf{A}}X_{A}(\omega )=-\infty$ for almost all $\omega$.
\end{corollary}

\begin{proof}
Clearly, we have
${\sup}_{\alpha\geq0}X_{A_{\alpha}}(\omega)\leq\sup_{A\in\mathbf{A}}X_{A}(\omega)$
(${\inf}_{\alpha\geq0}X_{A_{\alpha}}(\omega)\geq\break{\inf}_{A\in
\mathbf{A}}X_{A}(\omega))$ for all strictly increasing and continuous
sequences $\{A_{\alpha}\}_{\alpha\in\lbrack0,\infty)}$ in
$\mathbf{A(u)}$. By Theorem \ref{thm1} the process $X^{A}$ is an
$S_{[0,\infty)}$-time-changed Brownian motion for all strictly
increasing and continuous sequences $\{A_{\alpha
}\}_{\alpha\in\lbrack0,\infty)}$ in $\mathbf{A(u)}$; therefore,
$+\infty ={\sup}_{\alpha\geq0}X_{A_{\pi(\alpha)}}(\omega)\leq
{\sup}_{A\in\mathbf{A}}X_{A}(\omega)$ for almost all $\omega$ ($-\infty
\inf_{\alpha\geq0}X_{A_{\pi(\alpha)}}(\omega)\geq
\inf_{A\in\mathbf{A}}X_{A}(\omega)$ for almost all $\omega$).
\end{proof}

Let $B\in\mathbf{A(u)}$.
\begin{enumerate}
\item Let $A=\{A_{n}\}$ be an increasing sequence in $\mathbf{A}$.
    We write $A_{n}\uparrow B$ (or, in short notation, $A\uparrow
    B$) if $A_{n}\neq B$ for all $n$ and
    $\overline{{\xch{\bigcup}{\cup}}_{n}A_{n}}=B$.

\item Let $A=\{A_{n}\}$ be a decreasing sequence in $\mathbf{A}$.
    We write $A_{n}\downarrow B$ (or, in short notation,
    $A\downarrow T$) if $A_{n}\neq B$ for all $n$ and
    ${\xch{\bigcap}{\cap}}_{n}A_{n}=B$.
\end{enumerate}

\begin{corollary}[Strong law of large numbers and unboundedness]\label{cor7}
Let $\sigma\in M(\mathbf{A})$, and let
$X=\{X_{A}:A\in\mathbf{A}\}$ be a
set-indexed Brownian motion with variance $\sigma$. Then
\begin{enumerate}
\item[\rm(a)] $\varliminf_{A\uparrow T}\frac {X_{A}}{\sigma_{A}}=0$
    for almost all $\omega$, and for all sequences $A_{n}\uparrow
    T$, $\{A_{n}\}\in\mathbf{A}$.

\item[\rm(b)] $P[{\varliminf}_{A\uparrow T}X_{A}=-\infty
    ]=P[\varlimsup_{A\uparrow T}X_{A}=\infty]=1$.
\end{enumerate}
\end{corollary}

\begin{proof}
Let $A_{n}\uparrow T$. By Theorem \ref{thm1} and Remark
\ref{rem1} there exists a strictly increasing and continuous sequence
$\{B_{\alpha}\}_{\alpha\in\lbrack0,\infty)}\in\mathbf{A(u)}$ such
that $X^{B}$ is an $S_{[0,\infty)}$-time-changed Brownian motion when
$A_{n}=B_{n}$ . (In other words, there exist a strictly increasing and
continuous sequence $\{B_{\alpha}\}_{\alpha\in\lbrack0,\infty)}\in
\mathbf{A(u)}$ and $\{\alpha_{n}\}\in\lbrack0,\infty)$ such that
$X^{B,\pi}$ is a Brownian motion when $A_{n}=B_{n}=B_{\pi(\alpha
_{n})}$ and $\pi\in S_{[0,\infty)}$). Then
\begin{enumerate}
\item[\rm(a)] $\varliminf_{A\uparrow T}\frac {X_{A}}{\sigma_{A}}=
    \varliminf_{n\rightarrow\infty}\frac
{X_{_{n}}^{B,\pi}}{\sigma(B_{\pi(n)})}=\varliminf_{\alpha_{n}\rightarrow\infty
}\frac{X_{\alpha_{n}}^{B}}{\alpha_{n}}\overset{\ast}{=}0$ for
almost all $\omega$ (for the equality~\(\overset{\ast}{=}\), see
\cite{BoSa}).
\item[\rm(b)] $P[\varliminf_{A\uparrow
    T}X_{A}=-\infty]=P[\varliminf_{\alpha_{n}\rightarrow\infty}X_{\alpha
    _{n}}^{B}=-\infty]\overset{\ast\ast}{=}1$ and
    $P[\varlimsup_{A\uparrow
    T}X_{A}=\infty]=P[\varlimsup_{\alpha_{n}\rightarrow
\infty}X_{\alpha_{n}}^{B}=\infty]\overset{\ast\ast }{=}1$ (for the
equality \(\overset{\ast\ast}{=}\), see \cite{Fr}).\qedhere
\end{enumerate}
\end{proof}

\begin{corollary}[Law of iterated logarithm]\label{cor8}
Let $\sigma\in M(\mathbf{A})$, and let
$X=\{X_{A}:A\in\mathbf{A}\}$
be a set-indexed Brownian
motion with variance $\sigma$. Then
\[
\varliminf_{A\uparrow T}\frac{X_{A}}{\sqrt{2\sigma_{A}\ln\ln(\sigma_{A})}}=-1
\]
and
\[
\varlimsup_{A\uparrow T}\frac{X_{A}}{\sqrt{2\sigma_{A}\ln\ln(\sigma_{A})}}=1
\]
for almost all $\omega$ and for all $A\uparrow T$.
\end{corollary}

\begin{proof}
Let $A_{n}\uparrow T$. By Theorem \ref{thm1} and Remark \ref{rem1}
there exists a strictly increasing and continuous sequence
$\{B_{\alpha}\}_{\alpha\in\lbrack0,\infty)}\in\mathbf{A(u)}$ such that
$X^{B}$ is an $S_{[0,\infty)}$-time-changed Brownian motion when
$A_{n}=B_{n}$. (In other words, there exist a strictly increasing and
continuous sequence $\{B_{\alpha}\}_{\alpha\in\lbrack0,\infty)}\in
\mathbf{A(u)}$ and $\{\alpha_{n}\}\in\lbrack0,\infty)$ such that
$X^{B,\pi}$ is a Brownian motion when $A_{n}=B_{n}=B_{\pi(\alpha
_{n})}$ and $\pi\in S_{[0,\infty)}$)$.$ Then
\[
\varliminf_{A\uparrow T}
\frac{X_{A}}{\sqrt{2\sigma_{A}\ln\ln(\sigma_{A})}}
=
\varliminf_{\alpha_{n}\rightarrow\infty}
\frac{X_{\alpha_{n}}^{B,\pi}}{\sqrt{2\alpha_{n}\ln\ln (\alpha_{n})}}
\overset{\ast}{=}-1
\]
and
\[
\varlimsup_{A\uparrow T}
\frac{X_{A}}{\sqrt{2\sigma_{A}\ln\ln(\sigma_{A})}}
=
\varlimsup_{\alpha_{n}\rightarrow\infty}
\frac{X_{\alpha_{n}}^{B,\pi}}{\sqrt{2\alpha_{n}\ln\ln(\alpha_{n})}}\overset{\ast}{=}1
\]
for almost all $\omega$ (for the equality \(\overset{\ast}{=}\), see \cite{BoSa}).
\end{proof}

\begin{corollary}[The zero set]\label{cor9}
Let $\sigma\in M(\mathbf{A})$, and let $X=\{X_{A}:A\in\mathbf{A}\}$
be a set-indexed
Brownian motion with variance $\sigma$. Let $\omega\in\varOmega$ and
set
$Z_{\omega}=\{A\in\mathbf{A}:X_{A}(\omega)=0\}$; then the set
$Z_{\omega}$ is uncountable and
without monotone accumulation sets.
\end{corollary}

(The set $A\in\mathbf{A}$ is said to be a monotone
accumulation set if there exists an increasing (decreasing) sequence
$\{A_{n}\}\in Z_{\omega}$ such that $A_{n}\neq A$ and
\mbox{$A_{n}\uparrow A$ ($A_{n}\downarrow A)$}.)

\begin{proof}
From Theorem \ref{thm1} we conclude that the process $X^{A}$ is an
$S_{[0,\infty)}$-time-changed Brownian motion for all strict increasing
and continuous sequences $\{A_{\alpha}\}_{\alpha\in\lbrack0,\infty )}$
in $\mathbf{A(u)}$ (in other words, for all strict increasing and
continuous sequences $\{A_{\alpha}\}$ in $\mathbf{A(u)}$, there exists
$\pi\in S_{[0,\infty)}$ such that $X^{\pi,A}=\{X_{\pi\ast
A_{\alpha}}:\alpha\in\lbrack0,\infty)\}=\{X_{A_{\pi(\alpha
)}}:\alpha\in\lbrack0,\infty)\}$ is a Brownian motion). Then (see
\cite{BoSa}) the set $Z_{\omega}^{A}=\{\alpha\geq0:X_{\alpha}^{A,\pi
}(\omega)=0\}$ is uncountable and without monotone accumulation sets.
Thus, $Z_{\omega}$ is uncountable. The set $Z_{\omega}$ is without
monotone accumulation sets; if not, then there exists an increasing
sequence $\{A_{n}\}\in Z_{\omega}$ such that $A_{n}\neq A$ and
$A_{n}\uparrow A$, so that based on Theorem \ref{thm1}, it is easy to
see that there exists a strictly increasing and continuous sequence
$\{B_{\alpha }\}_{\alpha\in\lbrack0,\infty)}\in\mathbf{A(u)}$ such that
$X^{B}$ is an $S_{[0,\infty)}$-time-changed Brownian motion when
$A_{n}=B_{n}$. (In other words, there exist a strictly increasing and
continuous sequence $\{B_{\alpha}\}_{\alpha\in\lbrack0,\infty)}\in
\mathbf{A(u)}$ and $\{\alpha_{n}\}\in\lbrack0,\infty)$ such that
$X^{B,\pi}$ is a Brownian motion when $A_{n}=B_{n}=B_{\pi(\alpha
_{n})}$ and $\pi\in S_{[0,\infty)}$). Then the set $Z_{\omega
}^{B,\pi}=\{\alpha\geq0:X_{\alpha_{n}}^{B,\pi}(\omega)=0\}$ has a
monotone accumulation set, which is a contradiction (see \cite{BoSa}).
(In the same way, we can proceed in the case where $A_{n}\downarrow
A$).
\end{proof}

\section{Sequence-independent variation}

\begin{definition}\label{def10}
Let $\sigma$ be a positive and continuous measure in $\mathbf{A}$
(Radon measure). For $A\in\mathbf{A}$ and $\varepsilon
>0$, we define $D_{A}^{\varepsilon}=\{B\in\mathbf{A:}$ $A\subseteq
B$, $\sigma(B\backslash A)=\varepsilon\}$. Assume that
$D_{A}^{\varepsilon}\neq\emptyset$ and let $A^{\varepsilon}$ be an
element in $D_{A}^{\varepsilon}$.
\end{definition}

Hereafter, we assume that the space $T$ has a positive
and continuous measure~$\sigma$ in $\mathbf{A}$ such that for all
$A\in\mathbf{A}$, there exists $A^{\varepsilon}$ such that $\sigma
(A^{\varepsilon}\backslash A)=\varepsilon$ for all $\varepsilon>0$.

The classical example for definition is $T=\Re_{+}^{2}$ and $\mathbf
{A}=\mathbf{A}(\Re_{+}^{2})$ when $\sigma$ is a~Lebesgue or Radon
measure.\goodbreak

\begin{definition}\label{def11}
Let $X=\{X_{A}:A\in\mathbf{A}\}$
be a set-indexed stochastic process.
\begin{enumerate}
\item[\rm(a)] $X$ is said to have $\sigma$-stationary increments if
\[
X_{A_{1}^{\varepsilon}}-X_{A_{1}}\overset{d}{=}\cdots\overset{d}{=}
X_{A_{n}^{\varepsilon}}-X_{A_{n}}
\]
for all $\{A_{i}\}_{i=1}^{n}\in \mathbf{A}$, all $\varepsilon>0$, and all $A_{i}^{\varepsilon}\in
D_{A_{i}}^{\varepsilon}$ (the notation $\overset{d}{=}$ means the
equality in distribution).

\item[\rm(b)] $X$ is said to have independent increments if
$X_{C_{1}},\dots,X_{C_{n}}$ are independent random variables whenever
$C_{1},\dots,C_{n}$ are disjoint sets in $\mathbf{C}$.
\end{enumerate}
\end{definition}

Let $X=\{X_{t}:t\geq0\}$ be a square-integrable
martingale. It is known that we can associate with $X$ a unique
predictable process, denoted $\langle X\rangle$, such that
$X^{2}-\langle X\rangle$ is a martingale. Little is known in the
set-indexed case. However, the concept of increasing path allows us to
study such processes.

\begin{definition}\label{def12}
A square-integrable set-indexed martingale $X=\{X_{A}:A\in\mathbf{A}\}$
is said to have sequence-independent variation (s.i.v.) on
$\mathbf{A(u)}$ (or path-independent variation (p.i.v.)) if for any
strict increasing continuous sequences $\{A_{\alpha}\}_{\alpha\in
\lbrack a,b]}$ and $\{B_{\beta }\}_{\beta\in\lbrack a,b]}$ in
$\mathbf{A(u)}$ with $A_{a}=B_{a}$ and $A_{b}=B_{b}$, we have $\langle
X^{A}\rangle(b)=\langle X^{B}\rangle (b)$.
\end{definition}

\begin{remarks}
(a) This definition of
s.i.v. on $\mathbf{A(u)}$ was introduced by Cairoli and Walsh in the
plane \cite{CaWa}. Here we extend it to the set-indexed framework.

(b) The definition and more details about $\langle X\rangle$, can be
found in \cite{IvMe}.
\end{remarks}

\begin{theorem}\label{thm2}
Let $X=\{X_{A}:A\in\mathbf{A(\Re
}_{+}^{2})\}$ be a set-indexed strong martingale-$\mathbf{\Re
}_{+}^{2}$. If $\sup_{A\in\mathbf{A}}E[X_{A}^{4}]<\infty$ or the
filtration $F$ is generated by a Brownian motion, then $X$ has s.i.v.
on $\mathbf{A(u)}$.
\end{theorem}

\begin{proof}
It suffices to prove that for all increasing continuous sequences
$\{A_{\alpha}\}_{\alpha\in\lbrack a,b]}$ in $\mathbf{A(u)}$ and all
$a\leq s\leq t\leq b$,
$E[X_{A_{t}}^{2}-X_{A_{s}}^{2}|F_{A_{s}}]=E[\langle X\rangle
_{A_{t}}-\langle X\rangle_{A_{s}}|F_{A_{s}}]$. If $\{A_{\alpha
}\}_{\alpha\in\lbrack a,b]}$ is an increasing continuous sequence, then
$A_{s}\subseteq A_{t}$. The set $A_{t}\backslash A_{s}$ can be divided
to $n$ disjoint rectangles $\{R_{i}\}_{i=1}^{n}$ such that
$A_{t}\backslash A_{s}={\xch{\bigcup}{\cup}}_{i=1}^{n}R_{i}$. If $X$ is a strong
martingale-$\mathbf {\Re}_{+}^{2}$, then
$E[X_{A_{t}}^{2}-X_{A_{s}}^{2}|F_{A_{s}}]=E[(X_{A_{t}}-X_{A_{s}})^{2}|F_{A_{s}}]
$ and, by the division,
$E[(X_{A_{t}}-X_{A_{s}})^{2}|F_{A_{s}}]=E[(X_{R_{1}}+\cdots
+X_{R_{n}})^{2}|F_{A_{s}}]. $ It is clear that
$F_{A_{s}}\subseteq\mathbf{G}_{R_{i}}^{\ast}$ and for all $i\neq
j,E[X_{R_{i}}X_{R_{j}}|F_{A_{s}}]=E[E[X_{R_{i}}X_{R_{j}}|\mathbf{G}_{R_{i}}^{\ast}]|F_{A_{s}}]$
$=E[X_{R_{i}}E[X_{R_{j}}|\mathbf{G}_{R_{i}}^{\ast}]|F_{A_{s}}]=0$;
therefore,
$E[X_{A_{t}}^{2}-X_{A_{s}}^{2}|F_{A_{s}}]=E[X_{R_{1}}^{2}+\cdots\allowbreak
+X_{R_{n}}^{2}|F_{A_{s}}]\overset{\#}{=}E[\langle
X\rangle_{R_{1}}+\cdots+\langle X\rangle _{R_{n}}|F_{A_{s}}]=E[\langle
X\rangle_{A_{t}}-\langle X\rangle _{A_{s}}|F_{A_{s}}]$ (for the
equality \(\overset{\#}{=}\), see \cite{CaWa}).
\end{proof}

\begin{theorem}\label{thm3}
Let $\sigma\in M(\mathbf{A})$. If
$X=\{X_{A}:A\in\mathbf{A}\}$ is a set-indexed square-integrable
outer-continuous process with independent and $\sigma$-stationary
increments, then $X$ has s.i.v.
\end{theorem}

For the proof, we need two auxiliary propositions.

\begin{proposition}\label{prop1}
If $\{A_{\alpha}\}_{\alpha \in\lbrack0,\infty)}$ is a strictly
increasing continuous sequence in $\mathbf{A(u)}$, then $X^{A}$ is a
$G$-time-changed right-continuous martingale with independent and
stationary increments. (In other words, for all strictly increasing
continuous sequences\vadjust{\goodbreak} $\{A_{\alpha}\}_{\alpha\in \lbrack0,\infty)}$ in
$\mathbf{A(u)}$, there exists $\pi\in G$ such that $X^{\pi,A}$ is a
right-continuous martingale with independent and stationary
increments).
\end{proposition}

\begin{proof}
The process $X$ is outer-continuous; therefore, $X^{A}$ is
right-continuous. Since $X$ has independent increments, it is a strong
martingale. In particular, $X$ is a martingale \cite{IvMe}. It is easy
to see that $X^{A}$ is a martingale for all strict continuous sequences
$\{A_{\alpha}\}_{\alpha\in\lbrack0,\infty)}$ in $\mathbf{A(u)}$.
Moreover, $X$ has $\sigma$-stationary increments; therefore, $X^{A}$ is
a $G$-time-changed right-continuous martingale with independent and
stationary increments for all strict continuous sequences $\{A_{\alpha
}\}_{\alpha\in\lbrack0,\infty)}$ in $\mathbf{A(u)}$.
\end{proof}

Now, for any increasing continuous sequences $\{A_{\alpha}\}_{\alpha\in
\lbrack0,\infty)}$ and $\{B_{\beta}\}_{\beta\in\lbrack0,\infty )}$ in
$\mathbf{A(u)}$, $X^{\pi_{1},A}$ and $X^{\pi_{2},B}$ are
right-continuous martingales with independent and stationary
increments. We recall that if $X=\{X_{t}:t\geq0\}$ is a
right-continuous martingale with independent and stationary increments
such that $E[X_{t}^{2}]<\infty$ for all~$t$, then $\langle X\rangle
_{t}=tE[X_{1}^{2}]$ for all $t$. Thus, $\langle X^{\pi_{1},A}\rangle $,
$\langle X^{\pi_{2},B}\rangle$ are deterministic; in particular,
$\langle X^{\pi_{1},A}\rangle(c)=\langle X^{\pi_{2},B}\rangle(c)$ when
$\pi_{1}\ast A_{c}=\pi_{2}\ast B_{c}$, and for all $0\leq t$, $\langle
X^{\pi_{1},A}\rangle(c)=\sigma(A_{\pi_{1}(t)})$.

\begin{proposition} \label{prop2}
If $\sigma(A_{c})=k$, then there exists a unique $s\in\Re_{+}$ such
that \mbox{$A_{\pi _{1}(s)}=A_{c}$} and $s=k$ for all a strict continuous
sequences $\{A_{\alpha}\}_{\alpha\in\lbrack0,\infty)}$ in
$\mathbf{A(u)}$.
\end{proposition}

\begin{proof}
It is clear that $\sigma(A_{\pi
_{1}(k)})=k$ from the definition of $\pi_{1}$ ($\pi_{1}(\alpha
)=\theta^{-1}(\alpha)$ when $\theta:\Re_{+}\rightarrow\Re
_{+}$, $\theta(\alpha)=\sigma(A_{\alpha})$); therefore, $A_{\pi
_{1}(k)}=A_{c}$ or there exists $r\neq k$ when $A_{\pi
_{1}(r)}=A_{c}$ and $A_{\pi_{1}(r)}\neq A_{\pi_{1}(k)}$ because of
strictly increasing continuous sequences. Without loss of
generality, we may assume that $A_{\pi_{1}(r)}\subset A_{\pi_{1}(k)}$
when
$\sigma(A_{\pi_{1}(r)})= \sigma(A_{\pi_{1}(k)})$, which is a
contradiction to $\sigma\in M(\mathbf{A})$.
\end{proof}

\begin{proof}[Proof of Theorem \ref{thm3}]
Let $\{A_{\alpha}\}_{\alpha
\in\lbrack0,\infty)}$ and $\{B_{\beta}\}_{\beta\in\lbrack
0,\infty)}$ in $\mathbf{A(u)}$ be two strictly
increasing continuous sequences. If  $A_{0}=B_{0}$ and $A_{c}=B_{c}$,
then
$\sigma(A_{c})=\sigma(B_{c})=k$.
From Proposition \ref{prop2} we get that if
$A_{c}=A_{\pi_{1}(k)},B_{c}=B_{\pi_{2}(k)}$, then $\langle
X^{A}\rangle(c)=\langle X^{\pi_{1},A}\rangle(k)=\langle X^{\pi
_{2},B}\rangle(k)=\langle X^{B}\rangle(c)$.
\end{proof}

\begin{theorem}\label{thm4}
Let $\sigma\in M(\mathbf{A})$, and
let $X=\{X_{A}:A\in\mathbf{A}\}$ be a set-indexed Brownian motion
with variance $\sigma$. Then $X$ has s.i.v, and $\langle X\rangle
_{A}=\sigma_{A}$ for all $A\in\mathbf{A}$.
\end{theorem}

\begin{proof}
Since $X$ is a set-indexed Brownian motion, by Theorem \ref{thm1}
$X^{A}$ is a $G_{[0,\infty)}$-time-changed Brownian motion for all
strictly continuous sequences
$\{A_{\alpha}\}_{\alpha\in\lbrack0,\infty)}$ in $\mathbf
{A}(\mathbf{u})$ (in other words, for all strictly increasing and
continuous sequences $\{A_{\alpha}\}_{\alpha\in\lbrack0,\infty)}$ in
$\mathbf{A(u)}, $ there exists $\pi\in S_{[0,\infty)}$ such that
$X^{\pi,A}=\{X_{\pi\ast A_{\alpha}}:\alpha\in\lbrack0,\infty
)\}=\{X_{A_{\pi(\alpha)}}:\alpha\in\lbrack0,\infty)\}$ is a Brownian
motion). For any increasing\ continuous sequences
$\{A_{\alpha}\}_{\alpha\in\lbrack 0,\infty)}$ and
$\{B_{\beta}\}_{\beta\in\lbrack0,\infty)}$ in $\mathbf{A}(\mathbf{u})$,
$X^{\pi_{1},A}$ and $X^{\pi_{2},B}$ are Brownian motions; therefore,
$\langle X^{\pi_{1},A}\rangle$ and $\langle X^{\pi_{2},B}\rangle$ are
deterministic; in particular, $\langle X^{\pi_{1},A}\rangle(c)=\langle
X^{\pi_{2},B}\rangle(c)$ when $A_{\pi_{1}(c)}=B_{\pi_{2}(c)}$, and for
all $0\leq t$, $\langle
X^{\pi_{1},A}\rangle(c)=\sigma(A_{\pi_{1}(t)})$.

Let $\{A_{\alpha}\}_{\alpha\in\lbrack0,\infty)}$
and $\{B_{\beta}\}_{\beta\in\lbrack0,\infty)}$ in $\mathbf{A(u)}$
be two increasing continuous sequences with $A_{0}=B_{0}$ and
$A_{c}=B_{c}$ then $\sigma(A_{\pi_{1}(c)})= \sigma(B_{\pi
_{2}(c)})=k$

Returning to the proof of Theorem \ref{thm4}, from Proposition \ref{prop2} in Theorem \ref{thm3} 
we get that if $A_{c}=A_{\pi_{1}(k)}$ and $B_{c}=B_{\pi_{2}(k)}$,
then $\langle X^{A}\rangle(c)=\langle X^{\pi_{1},A}\rangle
(k)=\langle X^{\pi_{2},B}\rangle(k)=\langle X^{B}\rangle(c)$ and, in
particular, $\langle X^{A}\rangle(t)=\sigma(A_{t})$ for all $0\leq
t$.

Now, if $D\in\mathbf{A}$, then there exist a continuous sequence
$\{A_{\alpha}\}_{\alpha\in\lbrack0,\infty)}$ in $\mathbf
{A}(\mathbf{u})$ and $d\in\Re_{+}$ such that $A_{d}=D$. Then, by
Theorem \ref{thm1}, $X^{\pi,A}$ is a $G_{[0,\infty)}$-time-changed
Brownian motion, and $\langle X\rangle_{D}=\langle
X^{\pi,A}\rangle_{d}=\langle X^{\pi
}\rangle_{A_{d}}=\sigma_{A_{d}}=\sigma_{D}$.
\end{proof}

%

\end{document}